\documentclass[12pt,leqno]{article}

\usepackage{fullpage}
\usepackage{amssymb}
\usepackage{amsmath}
\usepackage{amsthm}
\usepackage[latin2]{inputenc}
\usepackage{amssymb}
\usepackage{paralist}
\usepackage{amsmath, calc}
\usepackage{amsfonts}
\usepackage{amsthm}
\usepackage{fullpage}
\usepackage{enumerate}

\newtheorem{theorem}{Theorem}
\newtheorem{problem}{Problem}

\newtheorem{corollary}{Corollary}

\newtheorem{remark}{Remark}
\begin{document}
\title{\Large On minimal additive complements of integers}
\date{}
\author{\large S\'andor Z. Kiss,$^{1}$\footnote{
Email:~kisspest@cs.elte.hu. This research was supported by the
National Research, Development and Innovation Office NKFIH Grant
No. K115288. }~ Csaba S\'andor,$^{1}$\footnote{
Email:~csandor@math.bme.hu. This author was supported by the OTKA
Grant No. K109789. This paper was supported by the J\'anos Bolyai
Research Scholarship of the Hungarian Academy of Sciences.}
~and~Quan-Hui Yang$^{2}$\footnote{Email:~yangquanhui01@163.com.
This author was supported by the National Natural Science
Foundation for Youth of China, Grant No. 11501299, the Natural
Science Foundation of Jiangsu Province, Grant Nos.
BK20150889,~15KJB110014 and the Startup Foundation for Introducing
Talent of NUIST, Grant No. 2014r029.} }
\date{} \maketitle
 \vskip -3cm
\begin{center}
\vskip -1cm { \small 1. Institute of Mathematics, Budapest
University of Technology and Economics, H-1529 B.O. Box, Hungary}
 \end{center}

 \begin{center}
{ \small 2. School of Mathematics and Statistics, Nanjing University of Information \\
Science and Technology, Nanjing 210044, China }
 \end{center}

\begin{abstract} Let $C,W\subseteq \mathbb{Z}$. If $C+W=\mathbb{Z}$, then the set $C$ is
called an additive complement to $W$ in $\mathbb{Z}$. If no proper
subset of $C$ is an additive complement to $W$, then $C$ is called
a minimal additive complement. Let $X\subseteq \mathbb{N}$. If
there exists a positive integer $T$ such that $x+T\in X$ for all
sufficiently large integers $x\in X$, then we call $X$ eventually
periodic. In this paper, we study the existence of a minimal
complement to $W$ when $W$ is eventually periodic or not. This
partially answers a problem of Nathanson.

\end{abstract}



\section{Introduction}

Let $\mathbb{N}$ denote the set of nonnegative integers and $\mathbb{Z}$ be the set of integers. For $A,B\subseteq \mathbb{Z}$ and $k \in \mathbb{Z}$, let $A+B=\{a+b:~a\in A,~b\in B\}$ and $kA = \{ka: a \in A\}$.
If $A+B=\mathbb{Z}$, then $A$ is called an additive complement to
$B$ in $\mathbb{Z}$. If no proper subset of $A$ is a complement to
$B$, then $A$ is called a minimal complement to $B$ in
$\mathbb{Z}$.

It is easy to see that if $A\subseteq \mathbb{Z}$ is a (minimal) complement to $B \subseteq \mathbb{Z}$, then $A$ is also a (minimal) complement to $B+d$, $d\in \mathbb{Z}$, where $B+d=\{ b+d:b\in B\} $.

In 2011, Nathanson \cite{nath} proved the following theorem.

\noindent{\bf  Nathanson's theorem} {\rm (See \cite[Theorem
8]{nath}).} {\em  Let $W$ be a nonempty, finite set of integers.
In $\mathbb{Z}$, every complement to $W$ contains a minimal
complement to $W$.}

In the same paper, Nathanson also posed the following problem.

\noindent{\bf  Problem} {\rm (See \cite[Problem 11]{nath}).} Let
$W$ be an infinite set of integers. Does there exist a minimal
complement to $W$? Does there exist a complement to $W$ that does not
contain a minimal complement?


For the second part of the above problem, in 2012, Chen and Yang \cite{chenyang} gave two infinite sets $W_1$ and $W_2$ of integers such that there exists a complement to $W_1$ that does not contain a minimal complement and every complement to $W_2$ contains a minimal complement.

For the first part of the above problem, in 2012, Chen and Yang \cite{chenyang} proved the following results.

\noindent{\bf  Theorem A} (See \cite[Theorem 1]{chenyang}).~{\em
Let $W$ be a set of integers with $\inf W=-\infty$ and $\sup
W=+\infty$. Then there exists a minimal complement to $W$.}

By Theorem A, now we only need to consider the cases $\inf W>-\infty$
or $\sup W<+\infty$. Without loss of generality, we may assume
that $\inf W>-\infty$.

\noindent{\bf  Theorem B} (See \cite[Theorem 2]{chenyang}).~{\em
Let $W=\{1=w_1<w_2<\cdots\}$ be a set of integers and
$$\overline{W}=(\mathbb{Z} \cap (0, +\infty)) \setminus W=\{\overline{w_1}
<\overline{w_2}<\cdots\}.$$

(a) If $\limsup_{i\rightarrow +\infty}(w_{i+1}-w_i)=+\infty$, then
there exists a minimal complement to $W$.

(b) If $\lim_{i\rightarrow
+\infty}(\overline{w_{i+1}}-\overline{w_i})=+\infty$, then there
does not exist a minimal complement to $W$.}

Let $W=\cup_{k=0}^{\infty} [10^k,2\times 10^k]$. Then it is clear
that both $\limsup_{i\rightarrow +\infty}(w_{i+1}-w_i)=+\infty$
and $\limsup_{i\rightarrow
+\infty}(\overline{w_{i+1}}-\overline{w_i})=+\infty$ hold. Hence
$\lim_{i\rightarrow
+\infty}(\overline{w_{i+1}}-\overline{w_i})=+\infty$ in Theorem B
(b) cannot be changed to $\limsup_{i\rightarrow
+\infty}(\overline{w_{i+1}}-\overline{w_i})=+\infty$.

In this paper, we will give further results on Nathanson's problem
and deal with some sets $W$ do not satisfy the conditions of
Theorem B.

First we give some definitions. Let $S\subseteq \mathbb{N}$. Denote by $S\mod m$ the set of residues of $S$ modulo $m$, i.e.,
$$S\mod m =\{ r: r\in \{0,1,\dots ,m-1\}, r\equiv s~(\text{mod}~m)\text{ for some }s\in S\}.$$
Let $X\subseteq \mathbb{N}$. If
there exists a positive integer $T$ such that $x+T\in X$ for all
$x\in X$, then we call $X$ {\em{periodic with period $T$}}. If
$X\cup C$ is a periodic set for some finite set $C\subseteq
\mathbb{N}$, then we call $X$ {\em{quasiperiodic}}. If there exists
a positive integer $T$ such that $x+T\in X$ for all sufficiently
large integers $x\in X$, then we call $X$ {\em {eventually
periodic with period $T$}}. Clearly, a periodic set must be
quasiperiodic and a quasiperiodic set must be eventually periodic.
If $W$ is eventually periodic with $|\mathbb{N}\setminus W| = +\infty$, then both $\lim_{i\rightarrow
+\infty}(w_{i+1}-w_i)<+\infty$ and $\lim_{i\rightarrow
+\infty}(\overline{w_{i+1}}-\overline{w_i})<+\infty$ hold. Hence
$W$ does not satisfy the conditions of Theorem B.

Suppose that $W$
is an eventually periodic set and $m$ is a positive
period. By shifting a number, we may assume that $W$ has the
following structure:
\begin{eqnarray}\label{def1}W=(m\mathbb{N}+X_m)\cup
Y^{(0)}\cup Y^{(1)},\end{eqnarray}
where $X_m\subseteq \{0, 1, \dots{}, m-1\}, Y^{(0)}\subseteq \mathbb{Z}^{-},Y^{(1)}$ are finite sets with
$Y^{(0)}~\text{mod}~ m\subseteq
X_m$ and
$(Y^{(1)}~\text{mod}~m) \cap X_m = \emptyset$.

For example, if $W=\{2,4,7,8,9,12,13,17,18,22,23,27,28,\ldots\}$, then by shifting a number $5$, we may assume that
$$W=\{-3,-1,2,3,4,7,8,12,13,17,18,22,23,\ldots\}.$$ Hence $m=5$, $X_m=\{2,3\}$, $Y^{(0)}=\{-3\}$, $Y^{(1)}=\{-1,4\}$.

In this paper, we study that what conditions are needed to ensure the existence of a minimal complement to $W$. First we prove a sufficient condition.



\begin{theorem}\label{thm3} Let $W$ be defined in \eqref{def1}. If there exists a minimal complement
to $W$, then there exists $C\subseteq \{0,1,\dots ,m-1\}$ such that the
following two conditions hold:

(a) $C+(X_m\cup Y^{(1)}) \mod m=\{0,1,\dots ,m-1\}$;

(b) For any $c\in C$, there exists $y\in Y^{(1)}$ such that
$c+y\not \equiv c'+x~(\text{mod}~m)$ for any $c'\in C$ and $x\in X_m$.
\end{theorem}

\begin{remark}\label{remark2} By the proof of Theorem \ref{thm3},
we know that Theorem \ref{thm3} also holds when $Y^{(1)}$ is an
infinite set with $|Y^{(1)}\cap \mathbb{Z}^{-}|<+\infty$.
\end{remark}

Let $m=3$, $X_m=\{0\}$, $Y^{(1)}\subseteq 3\mathbb{N}+1$. By
Theorem \ref{thm3}, we have the following corollary.

\begin{corollary}\label{coro1} Let $Y\subseteq 3\mathbb{N}+1$ and
$W=3\mathbb{N}\cup Y$. Then there does not exist a minimal
complement to $W$.
\end{corollary}

\begin{remark}\label{remark2} We can choose an infinite set
$Y$ in Corollary \ref{coro1} such that $W$ is not eventually
periodic. Hence, there exists an infinite, not eventually periodic
set $W\subseteq \mathbb{N}$ such that $w_{i+1}-w_i\in \{1,2,3\}$
for all $i$, and there does not exist a minimal complement to $W$.
\end{remark}

\begin{remark}\label{remark1} If $W\subseteq \mathbb{N}$ is a
quasiperiodic set, then $Y^{(1)}=\emptyset$ and the condition (b)
in Theorem \ref{thm3} does not hold. Hence there does not exist a
minimal complement to $W$.
\end{remark}

In the next step we prove a necessary condition.

\begin{theorem}\label{thm4} Let $W$ be defined in \eqref{def1}. Suppose that there exists $C\subseteq
\{0,1,\dots ,m-1\}$ such that the following two conditions hold:

(a) $C+(X_m\cup Y^{(1)})\mod m=\{0,1,\dots ,m-1\}$;

(b) For any $c\in C$, there exists $y\in Y^{(1)}$ such that $c+y\not \equiv c'+x~(\text{mod}~m)$ for
any $c'\in C \setminus \{c\}$ and $x\in X_m \cup Y^{(1)}$.

Then there exists a minimal complement to $W$.
\end{theorem}

By Theorems \ref{thm3} and \ref{thm4}, we have the following
corollary.

\begin{corollary} Let $W=(m\mathbb{N}+X_m)\cup Y^{(0)}\cup \{a\},$
where $X_m\subseteq \{0,1,\dots ,m-1\},~(Y^{(0)}\mod m)\subseteq X_m$ and $a\not \equiv x~(\text{mod}~m)$ if $x\in X_m$. Then there exists a minimal complement to $W$
if and only if there exists a subset $C\subseteq \{0,1,\dots ,m-1\}$ such that:

(a) $C+(X_m\cup\{ a\})\mod m =\{0,1,\dots ,m-1\}$;

(b) For any $c\in C$, $c+a\not \equiv c'+x~(\text{mod}~m)$, where $c'\in C \setminus \{c\}$ and $x\in X_m$.
\end{corollary}

We see that Theorems \ref{thm3} and \ref{thm4} transfer
Nathanson's problem into a finite modulo version when $W$ is an eventually periodic set.
In the next theorem, we give a sufficient and necessary condition, but we
cannot bound the module.

\begin{theorem}\label{thm5} Let W be defined in (1). There exists a minimal complement to $W$
if and only if there exists $T\in \mathbb{Z}^{+},m\mid T$,
and $C\subseteq \{0,1,\dots ,T-1\}$ such that

(a) $C+(X_T\cup Y^{(1)})\mod T=\{0,1,\dots ,T-1\}$, where $X_{T} = \cup_{i=0}^{\frac{T}{m}-1}\{im+X_{m}\}$;

(b) for any $c\in C$, there exists $y\in Y^{(1)}$ such that
$c+y\not \equiv c'+x~(\text{mod}~T)$ for any $c'\in C \setminus \{c\}$ and $x\in X_T\cup Y^{(1)}$.
\end{theorem}

Finally, as a complement to Remark \ref{remark2}, we give the
following theorem.

\begin{theorem}\label{thm2} There exists an infinite, not
eventually periodic set $W\subseteq \mathbb{N}$ such that
$w_{i+1}-w_i\in \{1,2\}$ for all $i$ and there exists a minimal
complement to $W$.
\end{theorem}

Now we pose two problems for further research.

\begin{problem} We know that Theorem \ref{thm3} also holds when
$Y^{(1)}$ is infinite. Is Theorem \ref{thm4} also true when
$Y^{(1)}$ is infinite?
\end{problem}

\begin{problem} Does there exist an explicit formula for the upper bound of $T$ in Theorem 3
using $m,Y^{(0)}$ and $Y^{(1)}$?
\end{problem}

\section{Proofs}

%
%
%

\begin{proof}[Proof of Theorem \ref{thm3}] Suppose that $D$ is a
minimal complement to $W$. For $i\in \{0,1,\ldots,m-1\}$, let
$D_i=\{d\in D:~d\equiv i~(\text{mod}~m)\}$ and
$$C=\{j:~0\le j\le m-1~\text{and}~|D_j\cap
\mathbb{Z}^{-}|=+\infty\}.$$ For any $t\in
\{0,1,\ldots,m-1\}\setminus C$, the set $\{d\in D:~d\equiv
t~(\text{mod}~m)\}+W$ does not contain any sufficiently small
negative integers. It follows from $D+W=\mathbb{Z}$ that
$C+W~\text{mod}~m =\{0,1,\dots ,m-1\}$. That is, $C+(X_m\cup Y^{(1)})~\text{mod}~m=\{0,1,\dots ,m-1\}$.

Next we shall prove (b). Suppose that there exists $c\in C$ such
that for any $y\in Y^{(1)}$ there exist $c'\in C$ and $x\in X_m$ with  $c+y\equiv c'+x~(\text{mod}~m)$. We take an integer $d\in D$ such that $d\equiv
c~(\text{mod}~m)$ and we shall prove that $D\setminus \{d\}$ is
also a complement to $W$. For any integer $n$, write $n=d'+w$,
where $d'\in D$ and $w\in W$.

Case 1. $d'\not=d$. Then $n=d'+w\in (D\setminus \{d\})+W$.

Case 2. $d'=d$.

Subcase 2.1. $(\{w\}~\text{mod}~m) \subseteq X_m$. In this case, there exists
a positive integer $k_0$ such that $w+km\in W$ for all integers
$k\ge k_0$. Since $|D_c\cap \mathbb{Z}^{-}|=+\infty$, it follows
that there exists an integer $k\ge k_0$ such that $d-km\in D$.
Hence $n=(d-km)+(w+km)$, where $d-km\in D\setminus \{d\}$ and
$w+km\in W$. That is, $n\in (D\setminus \{d\})+W$.

Subcase 2.2. $w\in Y^{(1)}$. Since $\{c+y\}~\text{mod}~m\subseteq C+X_m ~\text{mod}~m$
for any $y\in Y^{(1)}$ and $d\equiv c~(\text{mod}~m),~w\in
Y^{(1)}$, it follows that ~$\{n\}~\text{mod}~m=\{d+w\}~\text{mod}~m\subseteq C+X_m ~\text{mod}~m$. Hence
there exist a $c'\in D$ with $c'~(\text{mod}~m)\in C$ and $x\in W$
with $x~\text{mod}~m\in X_m$ such that $n\equiv
~c'+x~(\text{mod}~m)$. We choose a sufficiently large integer $k$
such that $c'-km\in D$,~$c'-km\not=d$ and $x+km\in W$. Hence
$n=(c'-km)+(x+km)$, where $c'-km\in D \setminus \{d\}$ and $x+km\in W$.

Hence, $(D\setminus \{d\})+W=\mathbb{Z}$ which contradicts the fact that $D$
is a minimal complement. Therefore, (b) holds.
\end{proof}

\begin{proof}[Proof of Theorem \ref{thm4}] Let
$C_1=C+X_m~\text{mod}~m$,~$C_2=\{0,1,\dots ,m-1\}\setminus C_1$,
$$C'=\{d\in\mathbb{Z}:~d\equiv c~(\text{mod}~m)~\text{for some}~c\in C\},$$
$$C_1'=\{d\in\mathbb{Z}:~d\equiv c~(\text{mod}~m)~\text{for some}~c\in C_1\},$$
$$C_2'=\{d\in\mathbb{Z}:~d\equiv c~(\text{mod}~m)~\text{for some}~c\in C_2\}.$$
By (a), we have $C'+W=\mathbb{Z}$. Since $C+X_m~\text{mod}~m=C_1$, it follows
that $$C'+(W\setminus
Y^{(1)})~\text{mod}~m=C+X_m~\text{mod}~m=C_1,$$ and so
$\left(C'+(W\setminus Y^{(1)})\right)\cap C_2'=\emptyset$. It follows from (b) that $C_2' \ne \emptyset$. Noting
that $C'+W=\mathbb{Z}$, we have $C'+Y^{(1)}\supseteq C_2'$. Since
$Y^{(1)}$ is a finite set, by the proof of Nathanson's theorem
(See \cite[Theorem 4, page 2015]{nath}), there exists $D'\subseteq C'$ such
that $D'+Y^{(1)}\supseteq C_2'$ and for any $d\in D'$,
$$(D'\setminus \{d\})+Y^{(1)}\not\supseteq C_2'.$$

Next we shall prove that $D'$ is a minimal complement to $W$.

For $i\in C$, let $D_i'=\{d\in D':d\equiv i~(\text{mod}~m)\}$. First we prove that $|D_i'\cap \mathbb{Z}^{-}|=+\infty$ for all $i\in
C$. Suppose that there exists a $j\in C$ such that $|D_j'\cap
\mathbb{Z}^{-}|<+\infty$. By (b), there exists a $y\in Y^{(1)}$ such
that $j+y\not \equiv c+x~(\text{mod}~m)$, where $c\in C\setminus \{j\}$, $x\in X_m\cup Y^{(1)}$ and so
$$D'+Y^{(1)}\not\supseteq \{d\in \mathbb{Z}:d\equiv  j+y~(\text{mod}~m)\}.$$

Noting that $(\{j+y\}~\text{mod}~m) \not\subseteq C+X_m ~\text{mod}~m=C_1$, we have $(\{j+y\}~\text{mod}~m)\subseteq
C_2$. It follows that $D'+Y^{(1)}\not\supseteq
C_2'$, a contradiction. Hence, $|D_i'\cap \mathbb{Z}^{-}|=+\infty$
for all $i\in C$.

Next we prove that $D^{'}$ is a complement.
For any integer $n\in C_1'$, by $C+X_m ~\text{mod}~m=C_1$, there exists $c\in C$
and $x\in X_m$ such that $n\equiv c+x~(\text{mod}~m)$. Since
$|D_c'\cap \mathbb{Z}^{-}|=+\infty$, there exists a sufficiently
small negative integer $d\in D_c'$ such that $n-d>0$. The congruences $n\equiv c+x ~(\text{mod}~m)$ and $d\equiv c ~(\text{mod}~m)$ imply that $n-d\equiv x ~(\text{mod}~m)$. Hence, $n-d\in m\mathbb{N}+X_m$ and so
$$n=d+(n-d)\in D_c'+(m\mathbb{N}+X_m)\subseteq D'+W.$$
Hence $C_1'\subseteq D'+W.$ On the other hand, $D'+W\supseteq
D'+Y^{(1)}\supseteq C_2'$. Therefore, $D'+W=\mathbb{Z}$.

Finally, we prove that $D^{'}$ is a minimal complement.
For any $d\in D'$,
we have $$\Big((D'\setminus \{d\})+(W\setminus
Y^{(1)})~\text{mod}~m\Big) \subseteq C+X_m ~\text{mod}~m=C_1.$$ It follows that $$\Big((D'\setminus
\{d\})+(W\setminus Y^{(1)})\Big)\cap C_2'=\emptyset,$$ and so
$(D'\setminus \{d\})+W\not\supseteq C_2'$. Hence $(D'\setminus
\{d\})+W\not=\mathbb{Z}$.

Therefore, $D'$ is a minimal complement to $W$.
\end{proof}

\begin{proof}[Proof of Theorem \ref{thm5}] Assume that the set $W$ satisfies the conditions of Theorem 3.
Applying Theorem 2 with $m = T$, it follows that $W$ has a
minimal complement.

Suppose that $W$ has a minimal complement $E$. We
will prove that there exist a positive integer $T$ and a set $C
\subseteq \{0,1, \dots{}, T-1\}$ satisfying the conditions of Theorem 3.

For $0 \le i < m$, let
\[
E^{-}_{i} = \{e: e < 0,e\in E, e \equiv i~(\text{mod}~m)\}.
\]
Let $0 \le i_{1} < i_{2} < \dots{} < i_{t} < m$ be the sequence of
indices with $|E^{-}_{i_{j}}| = \infty$. It is clear that there
exists an integer $N_{0}$ such that, $e \in E$ and $e \le
N_{0}$ imply that $e \in E_{i_{j}}$ for some $i_{j}$. It follows from Theorem 1 that $Y^{(0)} \cup Y^{(1)} \ne \emptyset$. Let
\[
y_{+} = \text{max}\{y: y \in Y^{(0)} \cup Y^{(1)}\},
\]
\[
y_{-} = \text{min}\{y: y \in Y^{(0)} \cup Y^{(1)}\},
\]
and $y_{0} = \max\{y_{+},-y_{-},y_{+}-y_{-}\}$. Let $\chi_{E}(k)$ denote the
characteristic function of the set $E$, i.e.,
\[
\chi_{E}(k) = \left\{
\begin{aligned}
1 \textnormal{, if } k \in E; \\
0 \textnormal{, if } k \notin E.
\end{aligned} \hspace*{3mm}
\right.
\]
Let $A = N_{0} + \text{min}\{0, y_{-}\}$. We consider the
following vectors:
\begin{align*}
\textbf{v}_{A} & = \left( \chi_{E}(A-y_0),\chi_{E}(A-y_0+1), \dots{}, \chi_{E}(A+y_0)\right), \\
\textbf{v}_{A-m} & = \left( \chi_{E}(A-m-y_0),\chi_{E}(A-m-y_0+1), \dots{}, \chi_{E}(A-m+y_0)\right), \\
&\mathrel{\makebox[\widthof{=}]{\vdots}} \\
\textbf{v}_{A-im} & = \left( \chi_{E}(A-im-y_0),\chi_{E}(A-im-y_0+1), \dots{}, \chi_{E}(A-im+y_0)\right) ,\\
&\mathrel{\makebox[\widthof{=}]{\vdots}}
\end{align*}
It is clear that there are infinitely many vectors
$\textbf{v}_{A}, \textbf{v}_{A-m}, \dots{},  \textbf{v}_{A-im},
\dots{},$ each of them has $2y_{0}+1$ coordinates, which are $0$ or
$1$. Since there are at most $2^{2y_{0}+1}$ different vectors,
by the pigeon hole principle, there exist a vector $\textbf{v}$
and an infinite sequence $0\le k_{1} < k_{2} <
\cdots{}$ such that $\textbf{v}_{A-k_{i}m} =  \textbf{v}$ for all positive integer $i$.

Define $L = A-k_{1}m$ and choose a sufficiently large integer $k_{i}$ such that
$k_{i}m - k_{1}m \ge y_{0}$
and $[A-k_{i}m, A-k_{1}m[ \cap E_{i_{j}} \ne \emptyset$ for every index $i_{j}$. Let $K = A - k_{i}m$,
$T = L-K$ and
\[
C = \{l: K \le l < L, ~ l \in E\}~\text{mod}~T.
\]

Now we shall prove that such an integer $T$ and a set $C$ satisfy the conditions of Theorem 3.
By definitions, $K$ and $L$ have the following properties.
\begin{equation}
L \le N_{0} + \text{min}\{0, y_{-}\},
\end{equation}
\begin{equation}
L-K\ge y_{0},
\end{equation}
\begin{equation}
m \mid L - K,
\end{equation}
\begin{equation}
\chi_{E}(K+i) = \chi_{E}(L+i)~ for~
-y_{0} \le i \le +y_{0},
\end{equation}
\begin{equation}
[K, L[ \cap E_{i_{j}} \ne \emptyset~\text{for all}~i_j.
\end{equation}

First we prove that $C + (X_{T} \cup Y^{(1)})~\text{mod}~T =
\{0,1,\ldots, T-1\}$, where $X_{T} = \bigcup_{i=0}^{\frac{T}{m}-1}\{im+X_{m}\}$.
For any integer $l$ with $K \le l < L$, there exist $e \in E$ and $w \in W$ such that $l = e + w$.
As $w \ge \text{min}\{0, y_{-}\}$,
it follows from (2) that $e = l - w < L - \text{min}\{0, y_{-}\}
\le N_{0}$, and so $e \in E_{i_{j}}^{-}$ for some integer $i_j$.
Suppose that $w \in  Y^{(0)} \cup Y^{(1)}$. Then we have $y_{-}
\le w \le y_{+}$ and $K-y_{+}
\le e =l-w< L-y_{-}$. We have three cases.

Case 1. $K - y_{+} \le e < K$. Noting that
$$K-y_0\le K-y_+\le e<K<K+y_0,$$
by (5), we have $e+(L-K)\in E$. By $K - y_{+} \le e < K$ and (3), it follows that
$K\le L-y_+\le e+(L-K)<L$. Let $c=e+(L-K)$. Then $c~\text{mod}~T \in C$ and $l\equiv c+w~(\text{mod}~T)$.
Hence $l~\text{mod}~T\in C + (X_{T} \cup Y^{(1)})~\text{mod}~T$.

Case 2. $K \le e < L$. It follows that $e~\text{mod}~T\in C$ and so
$l~\text{mod}~T \in C + (X_{T} \cup Y^{(1)})~\text{mod}~T$.

Case 3. $L \le e < L - y_{-}$. Noting that
$$L-y_0\le L\le e<L-y_{-}\le L+y_0,$$
by (5), we have $e - (L - K) \in E$. Since $L \le e < L - y_{-}$ and (3), it follows that
$K \le e - (L - K) < K - y_{-}<L$. Let $c=e-(L-K)$. Then $c~\text{mod}~T \in C$ and
$l\equiv c+w~(\text{mod}~T)$. Hence $l~\text{mod}~T\in C + (X_{T} \cup Y^{(1)})~\text{mod}~T$.

Suppose that $w \in T\mathbb{N} + X_{T}$. Since $w \ge 0$ and $e = l - w < L \le N_{0}$, we have $e \in E_{i_{j}}^{-}$ for some integer
$i_j$, and so $e \equiv i_{j}~(\text{mod}~m)$. It follows from (6) that
there exists an integer $e^{'}$ such that $e^{'} \in E_{i_{j}}^{-}$, $K \le
e^{'} < L$ and $e^{'} \equiv i_{j}~(\text{mod}~m)$. Let $w
\equiv x~(\text{mod}~m)$, where $x\in X_m$. Obviously, $l
\equiv e + w \equiv i_{j} + x \equiv e^{'} + x~(\text{mod}~m)$. By
(4) there exists an integer $u$ with $0 \le u < \frac{T}{m}$
such that $l \equiv e^{'} + um + x~(\text{mod}~T)$. Thus $l~\text{mod}~T\in C + (X_{T} \cup Y^{(1)})~\text{mod}~T$.

Therefore, $C + (X_{T} \cup Y^{(1)})~\text{mod}~T=\{0,1,\ldots,T-1\}$.

In the next step we show that the second condition of Theorem 3
holds. For any integer $e\in E$ with $K \le e < L$, there
exists a $w \in W$ such that $e + w \ne e^{'} + w^{'}$ for any $e' \ne
e$, $e'\in E$ and $w'\in W$.

If $w \in (m\mathbb{N} + X_{m}) \cup Y^{(0)}$, by $e<L\le N_0$ and $e\in E$,
then there
exists a positive integer $s$ such that $e + w = (e - sm) + (w +
sm)$, where $e - sm \in E$ and $w + sm \in W$. This is a contradiction.

Now we may assume that $w \in Y^{(1)}$. It is enough to prove
that $e + w \not\equiv e^{'} + w^{'}~(\text{mod}~T)$ for any $e^{'}(\not=e)\in E$, $K \le
e^{'} < L$ and $w^{'} \in X_{T} \cup (Y^{(1)}~\text{mod}~T)$. Suppose that such $e'$ and $w'$ 
exist, i.e., $e+w\equiv e'+w'~(\text{mod}~T)$.

If $w^{'} \in X_{T}$, then $e + w = e^{'} + w^{'}
+ tT$ for some integer $t$. Hence there exists a positive integer $s$ such that
$e + w = (e^{'} - sm) + (w^{'} + sm + tT)$, where
$e^{'} - sm  \in E$ and $w^{'} + sm + tT \in W$. This is a contradiction.

Now we assume $w^{'} \in Y^{(1)}$. Then $K + y_{-} \le  e +
w, e^{'} + w^{'} < L + y_{+}$. By $T=L-K\ge y_0\ge y_+-y_-$, it follows that either $e + w = e^{'} +
w^{'}$ or $e + w = e^{'} + w^{'} +T$ or $e + w = e^{'} +
w^{'} - T$.

Case 1. $e + w = e^{'} + w^{'}$. Then we have $e =
e^{'}$ and $w = w^{'}$, a contradiction.

Case 2. $e + w = e^{'} +
w^{'} +T$. It follows that $K + y_{-} \le e^{'} + w^{'} < K +
y_{+}$, and so $$K-y_0\le K + y_{-} -  y_{+} \le e^{'} < K + y_{+} -
y_{-}\le K+y_0.$$
By (5), we have $e^{'} + T\in E$, and then $e + w$ has another representation $(e^{'} +T) + w^{'}$.
Therefore $w
= w^{'}$ and $e = e^{'} + T$, which contradicts with $K \le
e, e^{'} < L$.

Case 3. $e + w = e^{'} + w^{'} - T$. Then we have
$L + y_{-} \le e^{'} + w^{'} < L + y_{+}$. Thus $L-y_0\le L +
y_{-} -  y_{+} \le e^{'} < L + y_{+} - y_{-}\le L+y_0$. It follows from
(5) that $e^{'} - T \in E$ which implies that $e + w =
(e^{'} - T) + w^{'}$. Therefore $w = w^{'}$ and $e =
e^{'} -T$, which is a contradiction because $K \le e, e^{'}
< L$.

The proof of Theorem 3 is completed.
\end{proof}

\begin{proof}[Proof of Theorem \ref{thm2}]

By induction we can construct
$\{d_i\}_{i=1}^{\infty},~\{W_i\}_{i=1}^{\infty}$ and
~$\{c_i\}_{i=1}^{\infty}$ such that

(i) $d_1=-1,~W_1=\{1,2,\ldots,12\},~c_1=-3$;

(ii) $d_i$ is the largest negative integer $\not\in
W_{i-1}+\{c_1,c_2,\ldots,c_{i-1}\}$ for $i\ge 2$;

(iii) $c_i<d_i+2c_{i-1}$ for all $i\ge 2$;

(iv) for $i\ge 2$, let $W_i=W_{i-1}\cup
\left([-2c_{i-1},-2c_i-1]\setminus \cup_{j=1}^{i-1}
\{-c_i+d_{j}\}\right).$

Let $W=\cup_{i=1}^{\infty}W_i$ and $C=\{c_i\}_{i=1}^{\infty}$.

Now we prove that $C$ is a minimal complement to $W$.

First we prove $d_{i+1}-d_i\le -2$ for all integers $i\ge 1$.
Clearly $d_2=-3,~d_2-d_1=-2$. Suppose that $d_{i+1}-d_i\le -2$ for
all integers $i<k~(k\ge 2)$. Since
$$d_k=(d_k-c_k)+c_k,\quad d_k-1=(d_k-1-c_k)+c_k,$$
$$-2c_{k-1}\le d_k-1-c_k<d_k-c_k<-c_k+d_{k-1},$$
it follows that $d_k-c_k,~d_k-1-c_k\in W_k$ and then
~$d_k,d_k-1\in W_k+\{c_k\}$. Hence $d_{k+1}\le d_k-2$. By (iv), we
have $w_{j+1}-w_j\in \{1,2\}$. Since $d_k\rightarrow -\infty$, by (ii) we
have $(-\infty,9]\subseteq W+C$. For any integer $n\ge 10$, there
exists an $i$ such that $-c_{i-1}\le n<-c_i$. Hence
$$-c_i+d_1<-c_{i-1}-c_i\le n-c_i<-2c_i,$$
and so $n-c_i\in W_i$, that is, $n\in W_i+\{c_i\}$. Therefore,
$W+C=\mathbb{Z}$.

Next, we prove that the complement $C$ is minimal. For any
positive integer $i$, we write $d_i=c+w$ with $c\in C$ and $w\in
W$. Now we shall prove that $c=c_i$. By (iv), we have
$d_i-c_j\not\in W$ for all integers $j>i$. Hence $c\not=c_j$ for
all integers $j>i$. Since $-2c_{i-1}$ is the minimal value of
$W\setminus W_{i-1}$ and for any positive integers $j\le i-1$,
$d_i-c_j\le d_i-c_{i-1}<-2c_{i-1}$, it follows that
$d_i-c_j\not\in W\setminus W_{i-1}$ for all positive integer $j\le
i-1$. Noting that $d_i\not\in W_{i-1}+\{c_1,\ldots,c_{i-1}\}$, we
have $d_i\not\in W+\{c_1,c_2,\ldots,c_{i-1}\}$. Hence $c=c_i$.

Therefore, $C$ is a minimal complement to $W$. Furthermore, by
(iii), we can choose suitable $c_i$ such that $W$ is infinite and
not eventually periodic.
\end{proof}

\section{Acknowledgement} This work was done during the third author visiting to Budapest
University of Technology and Economics. He would like to thank Dr.
S\'andor Kiss and Dr. Csaba S\'{a}ndor for their warm hospitality. We would like
to thank the anonymous referee very much for the detailed comments.

\end{document}